%% file: main.tex
\documentclass[preprint,12pt,letterpaper,nopreprintline]{elsarticle}
\usepackage[margin=1in]{geometry}
\usepackage{lmodern}
\usepackage{amsmath,amssymb,amsthm,mathtools}
\usepackage{booktabs,microtype}
\usepackage[hidelinks]{hyperref}
\newtheorem{theorem}{Theorem}[section]
\newtheorem{lemma}[theorem]{Lemma}
\newtheorem{proposition}[theorem]{Proposition}
\newtheorem{corollary}[theorem]{Corollary}
\theoremstyle{definition}

\theoremstyle{remark}
\newtheorem{remark}[theorem]{Remark}
\newcommand{\F}{\mathbb F}
\newcommand{\one}{\mathbf 1}
\newcommand{\AGL}{\operatorname{AGL}}
\newcommand{\GL}{\operatorname{GL}}
\newcommand{\Var}{\operatorname{Var}}
\newcommand{\E}{\mathbb E}

\newcommand{\Span}{\operatorname{span}}
\journal{Journal of Combinatorial Theory, Series A}
\hypersetup{pdftitle={Four Special Directions over Prime Fields: Structure and Complete Classification},
pdfauthor={Qihang Wang},pdfsubject={Finite geometry and algebraic combinatorics},
pdfkeywords={special direction, finite affine plane, projection polynomial, Fourier support, affine classification}}
\begin{document}
\begin{frontmatter}
\title{Four Special Directions over Prime Fields:
Structure and Complete Classification}
\author[pku]{Qihang Wang\corref{corresponding}}
\cortext[corresponding]{Corresponding author.}
\ead{2501110049@stu.pku.edu.cn}
\affiliation[pku]{organization={School of Mathematical Sciences, Peking University},
city={Beijing},postcode={100871},country={China}}
\begin{abstract}
A direction of the affine plane $\F_p^2$ is special for a set if the lines
in that direction do not all meet the set in the same number of points.
We classify, for every prime $p\ge5$, the subsets of $\F_p^2$ with exactly
four special directions. There are $54$ affine-equivalence classes,
occurring only for $p=5,7,11,13,17$, and we give explicit representatives.
In particular, $17$ is the largest prime admitting such a set, which
answers a question of Ghidelli, Kiss and Somlai, and over $\F_{13}$ the
minimum size is $65$, attained by exactly two classes. The key structural
fact is that the projection polynomials of such a set share a vertex: the
set is either centrally symmetric about this point or differs from its
reflection by exchanging one full line with one empty line. We also show
that, over an odd prime field, a Boolean function on $\F_p^d$ with exactly
four active projective Fourier lines is pulled back from a plane, so the
classification extends to every dimension. The classification and the
exclusion of the primes $19\le p<2048$ are computer-assisted; the
structural results and the exclusion of larger primes are proved directly.
\end{abstract}
\end{frontmatter}
\input{introduction}
\input{fourier}
\input{quadratics}
\input{reflection}
\input{feasibility}
\input{threshold}
\input{classification}
\appendix
\input{obstruction13}
\input{algorithms}
\section*{Acknowledgments}
The author acknowledges Gewu Intelligence Lab for providing the collaborative
research environment in which this project was developed.
\section*{Funding}
This research did not receive any specific grant from funding agencies in the
public, commercial, or not-for-profit sectors.
\section*{Declaration of competing interest}
The author declares no competing interests.
\section*{Data and code availability}
The ancillary files accompanying this article contain the exact
enumeration and verification programs, machine-readable representatives,
and the finite variance-margin table. Their \texttt{README.md} specifies
the dependencies, execution command and expected outputs.
Section~\ref{sec:classification} gives the reconstruction rules, and
\ref{app:algorithms} describes the finite domains and comparison tests.
\section*{Declaration of generative AI and AI-assisted technologies in the manuscript preparation process}
During the preparation of this work the author used OpenAI GPT-5.6 Sol,
Anthropic Claude Fable 5 and Grok 4.6 to organize the arguments and to draft
and revise the manuscript; OpenAI GPT-6 Astra and OpenAI GPT-6.1 Sol to
integrate and revise the manuscript and to verify the literature and the
references; and Anthropic Claude Opus 5.5 to check the derivations and
citations line by line and to edit the language. The use of AI tools in the
research itself is described in \ref{app:ai}.
After using these tools, the author reviewed and edited the content as needed and takes full responsibility for the content of the published article.
\input{bibliography}
\end{document}

%% file: introduction.tex
\section{Introduction}
Let $p$ be a prime. A \emph{direction} of the affine plane $\F_p^2$ is a
one-dimensional linear subspace; its cosets form a parallel class of $p$
lines. A set $S\subseteq\F_p^2$ is \emph{equidistributed} in a direction if
all lines of that parallel class meet $S$ in the same number of points;
otherwise the direction is \emph{special} for $S$. When $|S|=p$, the special
directions are exactly the directions determined by pairs of points of $S$,
so this notion extends the classical direction problem of R\'edei
\cite{Redei}; see also Sz\H onyi \cite{Szonyi} for smaller sets, and
\cite{Ball,BBBSS} for the theory over arbitrary finite fields.
For larger sets the study was initiated by Ghidelli \cite{Ghidelli} and
continued by Kiss and Somlai \cite{KissSomlai}, who showed that no set has
exactly two special directions and classified the sets with exactly three.
Ghidelli, Kiss and Somlai \cite{GKS} proved that for every fixed $k\ge4$ no
set has exactly $k$ special directions once $p$ is sufficiently large.
In this paper we settle the case $k=4$ for every prime
$p\ge5$.

We consider ordinary subsets, without multiplicities, up to the action of
the affine group $\AGL(2,p)$. A line is \emph{full} if it is contained in
$S$ and \emph{empty} if it is disjoint from $S$. If $p\ge5$ and $S$ has
exactly four special directions, then at least $p-3\ge2$ directions are
equidistributed, so $|S|=np$ with $1\le n\le p-1$. For a nonzero linear form
$\ell$ on $\F_p^2$, the \emph{projection function}
$a_\ell(t)=|S\cap\{\ell=t\}|$ records the line counts in one parallel
class, and the \emph{projection polynomial} $q_\ell\in\F_p[t]$ is the unique
polynomial of degree less than $p$ with $q_\ell(t)\equiv a_\ell(t)\pmod p$
for all $t\in\F_p$. Kiss and Somlai \cite[Proposition~3.1]{KissSomlai}
proved that the projection polynomials of a set with $k$ special directions
have degree at most $k-2$. For four directions they are therefore at most
quadratic, and our first result shows that they are genuine quadratics
with a common vertex.

\begin{theorem}[Quadratic center and reflection]\label{thm:reflection}
Let $p\ge5$ be prime and let $S\subseteq\F_p^2$ have exactly four special
directions. The projection polynomial of $S$ along each special direction
has degree exactly two, and there is a unique point $c\in\F_p^2$ such that
$s\mapsto q_\ell(\ell(c)+s)$ is an even polynomial for every nonzero linear
form $\ell$. Moreover, exactly one of the following holds.
\begin{enumerate}
\item $S$ is centrally symmetric about $c$, that is, $2c-S=S$.
\item $S$ has exactly one full line $L$ and exactly one empty line $L'$,
these lines satisfy $L'=2c-L$, and
\begin{equation}\label{eq:reflection-main}
 \one_S(z)-\one_S(2c-z)=\one_L(z)-\one_{L'}(z)
 \qquad(z\in\F_p^2).
\end{equation}
Both $S\setminus L$ and $S\cup L'$ are centrally symmetric about $c$ and
have the same four special directions as $S$.
\end{enumerate}
In either case, $S$ is obtained from a nonempty set containing no full line
by adding at most two of its empty lines, all in one parallel class.
\end{theorem}

We call $c$ the \emph{quadratic center} of $S$. In the second case $S$ has
no center of symmetry at all, yet $c$ is still determined by the projection
functions. Consequently, every affine equivalence between two such sets
maps quadratic center to quadratic center, and after translating the
centers to the origin, affine classification becomes linear classification
(Corollary~\ref{cor:center-equivalence}).

\begin{theorem}[Complete classification]\label{thm:classification}
Over all primes $p\ge5$ together, the subsets of $\F_p^2$ with exactly four
special directions form $54$ affine-equivalence classes. They occur exactly
for $p=5,7,11,13,17$, with the numbers and cardinalities shown in
Table~\ref{tab:spectrum}. Representatives are the $27$ sets specified in
Table~\ref{tab:representatives} together with their complements.
\end{theorem}

\begin{table}[ht]
\centering
\caption{The affine classes of subsets of $\F_p^2$ with exactly four
special directions. An entry $n:m$ means that there are $m$ classes of
sets of size $np$. The last two columns count the classes in the first and
in the second case of Theorem~\ref{thm:reflection}. Since $p$ is odd,
a set and its complement never lie in the same class.}
\label{tab:spectrum}
\begin{tabular}{cclrr}
\toprule
$p$ & Classes & $n:$ number of classes & Central & Noncentral\\
\midrule
5  &12& $1:1,\ 2:5,\ 3:5,\ 4:1$ &8&4\\
7  &22& $2:2,\ 3:9,\ 4:9,\ 5:2$ &16&6\\
11 &12& $3:1,\ 4:2,\ 5:3,\ 6:3,\ 7:2,\ 8:1$ &10&2\\
13 &6 & $5:2,\ 6:1,\ 7:1,\ 8:2$ &6&0\\
17 &2 & $8:1,\ 9:1$ &2&0\\
\bottomrule
\end{tabular}
\end{table}

Theorem~\ref{thm:reflection} holds uniformly in $p$, but the proof of
Theorem~\ref{thm:classification} is computer-assisted. At each prime,
Theorem~\ref{thm:reflection} and an integer moment criterion
(Proposition~\ref{prop:criterion}) reduce the problem to an explicit finite
family of quadratic profiles. The computation reconstructs every set that
arises from a feasible profile, including the sets with full lines, and
sorts these sets into $\GL(2,p)$-orbits. The same criterion, combined with
an analytic variance bound, excludes all primes $p\ge19$. The programs and
data are included as ancillary files and are specified in
\ref{app:algorithms}.

Following \cite[Question~2]{GKS}, let $T_4$ be the least integer such that
no prime $p>T_4$ admits a subset of $\F_p^2$ with exactly four special
directions.

\begin{corollary}\label{cor:T4}
One has $T_4=17$. Over $\F_{13}$, the minimum size of a set with exactly
four special directions is $65$, and there are exactly two affine classes
of such sets of size $65$.
\end{corollary}

Ghidelli, Kiss and Somlai asked whether some prime $p\ge17$ admits a set
with exactly four special directions, or equivalently whether $T_4>13$
\cite[Question~3]{GKS}. Corollary~\ref{cor:T4} shows that $p=17$ does and
that no larger prime does. Table~\ref{tab:spectrum} also answers, for four
special directions, the questions raised by Kiss and Somlai in
\cite[Section~7]{KissSomlai}. First, for $p\ge5$, a set that is
equidistributed in exactly $p-3$ directions exists if and only if
$p\in\{5,7,11,13,17\}$. Second, Ghidelli's bound \cite[Theorem~1.3]{Ghidelli} shows that a
set of size $np$ with exactly four special directions satisfies
$n\ge\lceil(p-2)/3\rceil$. This bound is attained for $p=5,7,11$, but not for $p=13$ and
$p=17$, where the least sizes are $5p$ and $8p$ rather than $4p$ and $5p$.

The planar problem also governs higher dimensions. For
$f:\F_p^d\to\{0,1\}$, call a one-dimensional subspace of the dual space
\emph{active} if $\widehat f(\xi)\ne0$ for some nonzero vector $\xi$ in it.
Its nonzero vectors are the normals of a parallel class of hyperplanes, and
the line is active exactly when the hyperplane counts of $f$ in that class
are not constant (Section~\ref{sec:fourier}). For $d=2$, the active lines
are thus the normals of the special directions.

\begin{theorem}[Fourier dimension]\label{thm:dimension}
Let $p$ be an odd prime and $f:\F_p^d\to\{0,1\}$. If $f$ has at most four
active lines, they span a subspace of dimension at most two. If $f$ has
exactly four, then its group $K$ of translation periods has dimension $d-2$,
and $f$ is the pullback of a Boolean function on the plane $\F_p^d/K$ with
exactly four special directions. Two such functions are affinely equivalent
if and only if their quotient functions are.
\end{theorem}

Consequently, for every $d\ge2$ and every prime $p\ge5$,
Table~\ref{tab:spectrum} also classifies the subsets of $\F_p^d$ with
exactly four special parallel classes of hyperplanes, with all sizes
multiplied by $p^{d-2}$ (Corollary~\ref{cor:cylinders}).

\subsection{Relation to earlier work}
Kiss and Somlai \cite{KissSomlai} constructed sets with exactly four special
directions over $\F_5$, $\F_7$ and $\F_{11}$ attaining Ghidelli's bound, and
a $65$-point set over $\F_{13}$; they asked whether a $52$-point example
exists over $\F_{13}$ \cite[Section~6]{KissSomlai}. \ref{app:13} answers
this question without the full classification. Adriaensen, Csajb\'ok and
Weiner \cite{ACW} extended the three-direction classification to all finite
translation planes, and Adriaensen, Sz\H onyi and Weiner \cite{ASW} studied
multisets with few special directions in a modular setting and their
relation to small-weight codewords of Desarguesian planes.

The proof of nonexistence for large $p$ in \cite{GKS} combines the degree
bound with a reconstruction identity and with discrepancy or second-moment
estimates. For $k=4$, their explicit second-moment theorem applies when
$p\ge207769$ \cite[Remark following Theorem~1.3]{GKS}. We use the same
reconstruction identity and second-moment strategy. In the four-direction
case, all four projection polynomials of a set without full lines are genuine
quadratics. Completing each square expresses its projection function as the
least residue of $At^2+\delta$, up to a translation of the variable; its value distribution
is given exactly by the quadratic character. This yields Lemma~\ref{lem:profile-variance}, which
lowers the analytic cutoff to $2048$, and also an integer moment criterion
for reconstruction that makes the remaining primes and the classification
a finite exact computation.

For the higher-dimensional statement we use the Fourier framework of
Iosevich, Liu, Mayeli and Pakianathan \cite{ILMP}: Galois saturation of the
Fourier support of rational-valued functions, coarse bandwidth, and passage
to quotients (Theorem~2.3, Definition~2.4 and Corollary~4.2 of the arXiv
version). Combined with that corollary, Theorem~4.3 there already excludes
four active lines spanning a space of dimension four;
Theorem~\ref{thm:dimension} disposes of the remaining case of dimension
three. Its key step uses Boolean idempotence to
exclude frequencies with a unique representation as a sum of two active
frequencies, an obstruction also used over $\F_2$ by Mande and Sanyal
\cite{MandeSanyal}. Structure and testing of Fourier-sparse Boolean
functions on general abelian groups are studied by Chakraborty, Datta,
Dutta, Ghosh and Sanyal \cite{Chakraborty}. Lev \cite{Lev} gives an
uncertainty-principle approach to perfect directions. Cylinder-type
results for $p$-divisible sets in $\F_p^3$ \cite{Cylinder} concern
different hypotheses and linear combinations of cylinders, whereas
Theorem~\ref{thm:dimension} produces a common period space of a single
Boolean function.

\subsection{Organization}
Section~\ref{sec:fourier} proves Theorem~\ref{thm:dimension} by a secant
covering argument. Section~\ref{sec:quadratics} shows that the projection
polynomials of a set without full lines are genuine quadratics with a
common vertex, and Section~\ref{sec:reflection} treats full lines and
completes the proof of Theorem~\ref{thm:reflection}.
Section~\ref{sec:feasibility} characterizes the quadratic profiles that
arise from sets. Section~\ref{sec:threshold} excludes all primes $p\ge19$
and constructs the example over $\F_{17}$, and
Section~\ref{sec:classification} completes the classification.
\ref{app:13} gives a computer-free treatment of the $52$-point case
over $\F_{13}$, and \ref{app:algorithms} specifies the finite
computations.

%% file: fourier.tex
\section{Projective Fourier rigidity}\label{sec:fourier}
For $\xi,x\in\F_p^d$, write $\xi\cdot x$ for the standard pairing and
$e_p(t)=\exp(2\pi it/p)$ for $t\in\F_p$. We use the normalization
\begin{equation}\label{eq:fourier}
 \widehat f(\xi)=\sum_{x\in\F_p^d}f(x)e_p(-\xi\cdot x),
 \qquad
 f(x)=p^{-d}\sum_{\xi\in\F_p^d}\widehat f(\xi)e_p(\xi\cdot x).
\end{equation}
As in the introduction, a line $\F_p\xi$ of the dual space is
\emph{active} for $f$ if $\widehat f$ does not vanish identically on
$\F_p\xi\setminus\{0\}$. Let $U(f)$ be the linear span of the active lines.

\begin{lemma}[Scalar saturation]\label{lem:saturation}
If $f$ is rational-valued and $a\in\F_p^\times$, then
$\widehat f(a\xi)=\sigma_a(\widehat f(\xi))$, where $\sigma_a$ is the
automorphism of $\mathbb Q(e_p(1))$ with $\sigma_a(e_p(1))=e_p(a)$.
In particular, $\widehat f$ is nonzero at every nonzero vector of an active
line.
\end{lemma}
\begin{proof}
Apply $\sigma_a$ termwise to \eqref{eq:fourier}; it fixes the rational
values of $f$. Being a field automorphism, $\sigma_a$ preserves
nonvanishing.
\end{proof}
This is the Galois support principle of \cite[Theorem~2.3]{ILMP}.

For a nonzero form $\ell(x)=\xi\cdot x$, the hyperplane counts
$a(t)=\sum_{\ell(x)=t}f(x)$ satisfy $\widehat a(s)=\widehat f(s\xi)$ for
$s\in\F_p$. Hence $a$ is constant if and only if the line $\F_p\xi$ is
inactive. In other words, the active lines are exactly the normals of the
parallel classes of hyperplanes in which the counts of $f$ are not
constant.

\begin{lemma}[Secant cover]\label{lem:secant}
Let $f:\F_p^d\to\{0,1\}$ have exactly $k$ active lines, and suppose that a
two-dimensional subspace $P$ contains exactly two of them, $L_1$ and $L_2$.
Then every line of $P$ other than $L_1,L_2$ lies in $\Span(L_i,L_j)$ for
some pair of distinct active lines $L_i,L_j$ not contained in $P$.
Consequently,
\begin{equation}\label{eq:secant}
 p-1\le\binom{k-2}{2}.
\end{equation}
\end{lemma}
\begin{proof}
Suppose that a nonzero $h\in P\setminus(L_1\cup L_2)$ lies in no such span.
Since $h$ lies on no active line, $\widehat f(h)=0$. As $f^2=f$, the
convolution identity gives
\begin{equation}\label{eq:idempotence}
 0=\widehat f(h)=p^{-d}\sum_{\xi\in\F_p^d}\widehat f(\xi)\widehat f(h-\xi).
\end{equation}
A term is nonzero only if $\xi$ and $h-\xi$ are nonzero vectors on active
lines. The terms with $\xi=0$ or $\xi=h$ contain the factor
$\widehat f(h)=0$. Two vectors on the same active line cannot sum to $h$,
since $h$ lies on no active line. If one of $\xi,h-\xi$ lies on $L_1$ or
$L_2$, then the other lies in $P$, hence also on $L_1$ or $L_2$. Two active
lines outside $P$ are excluded by the choice of $h$. Hence the only nonzero
terms come from the unique decomposition $h=u+v$ with $u\in L_1$ and
$v\in L_2$; here $u,v\ne0$ because $h\notin L_1\cup L_2$. These two ordered
terms contribute $2\widehat f(u)\widehat f(v)$, which is nonzero by
Lemma~\ref{lem:saturation}. This contradicts \eqref{eq:idempotence}.

The plane $P$ contains $p+1$ lines, hence $p-1$ lines other than $L_1,L_2$.
If $L_i,L_j$ are distinct active lines not contained in $P$, then
$\Span(L_i,L_j)\ne P$, so it contains at most one line of $P$. There are
$\binom{k-2}{2}$ such pairs, which proves \eqref{eq:secant}.
\end{proof}

\begin{proof}[Proof of Theorem~\ref{thm:dimension}]
If $k\le2$ there is nothing to prove. Since $p$ is odd,
\eqref{eq:secant} fails whenever $3\le k\le4$, because then
$\binom{k-2}{2}\le1<p-1$. It therefore suffices to exhibit, in every
configuration of at most four active lines spanning a space of dimension at
least three, a plane containing exactly two of them. If three active lines are linearly independent, the span
of any two of them contains no third active line when $k=3$, and the same
holds for four linearly independent lines.

It remains to consider four active lines spanning a space of dimension
three. Choose representatives $e_1,e_2,e_3$ of three independent lines
and a representative $v$ of the fourth. Then $v$ has two or three nonzero
coordinates in this basis, since one nonzero coordinate would give a basis
line. If it has two, we may assume that $v\in\Span(e_1,e_2)$; then
$\Span(e_1,e_3)$ contains exactly two active lines. If it has three, then
$\Span(e_1,e_2)$ contains exactly two active lines. In either case
Lemma~\ref{lem:secant} gives a contradiction, so $\dim U(f)\le2$.

Now let
\[
 K=U(f)^\perp=\{t\in\F_p^d:\xi\cdot t=0\text{ for all }\xi\in U(f)\}.
\]
Since $\widehat f$ is supported on $U(f)$, Fourier inversion shows that
every $t\in K$ is a translation period of $f$. Conversely, if
$f(x+t)=f(x)$ for all $x$, then $(e_p(\xi\cdot t)-1)\widehat f(\xi)=0$ for
every $\xi$, so $t$ is orthogonal to every active line and $t\in K$.
With exactly four active lines we have $\dim U(f)=2$, hence $\dim K=d-2$.
The function induced on $\F_p^d/K$ is Boolean, because its value at a coset
is the common value of $f$ on that coset. The dual of $\F_p^d/K$ is
$K^\perp=U(f)$, and the Fourier coefficients of the quotient function are
those of $f$ on $U(f)$ divided by $p^{d-2}$. Hence the quotient function
has the same four active lines, that is, exactly four special directions.

An affine bijection carrying one such function to another carries the
period group of the first onto that of the second, and therefore induces an
affine equivalence of the quotient functions. Conversely, an affine
equivalence of the quotients lifts to the ambient spaces after choosing
complements of the two period groups and a linear isomorphism between them.
\end{proof}

\begin{remark}
Odd characteristic is necessary. The majority function on $\F_2^3$ has
exactly four active lines, spanned by $e_1,e_2,e_3$ and $e_1+e_2+e_3$,
and they span $\F_2^3$. In dimension $d>2$, for hyperplanes pulled back from
the quotient plane, the counts are $p^{d-2}$ times the corresponding planar line counts. The
quadratic analysis below is therefore carried out for the planar quotient,
not for the ambient counts modulo $p$; the planar quadratic center lifts to
a coset of $K$.
\end{remark}

%% file: quadratics.tex
\section{Quadratic projections and their common center}\label{sec:quadratics}
From now on $p\ge5$ and $S\subseteq\F_p^2$ has exactly four special
directions. Choose linear forms $\ell_1,\ldots,\ell_4$ defining their
parallel classes, and put
\[
 a_i(t)=|S\cap\{\ell_i=t\}|\quad(t\in\F_p),\qquad |S|=np.
\]
Each $a_i$ takes values in $\{0,\ldots,p\}$ and has average $n$.

\begin{lemma}[Incidence and variance]\label{lem:incidence}
For every $z\in\F_p^2$,
\begin{equation}\label{eq:incidence}
 \sum_{i=1}^4a_i(\ell_i(z))=3n+p\one_S(z).
\end{equation}
If $U$ is uniformly distributed on $\F_p$, then
\begin{equation}\label{eq:variance}
 \sum_{i=1}^4\Var(a_i(U))=n(p-n).
\end{equation}
\end{lemma}
\begin{proof}
Sum the numbers of points of $S$ on the $p+1$ lines through $z$. Each point
of $S\setminus\{z\}$ is counted once, and $z$, if it lies in $S$, is counted
$p+1$ times; the total is $np+p\one_S(z)$. Each of the $p-3$ nonspecial
lines through $z$ contributes $n$, which gives \eqref{eq:incidence}.

Let $Z$ be uniformly distributed on $\F_p^2$. For nonproportional forms
$\ell_i,\ell_j$, the map $z\mapsto(\ell_i(z),\ell_j(z))$ is a bijection of
$\F_p^2$, so $\ell_i(Z)$ and $\ell_j(Z)$ are independent and uniform.
Hence the variance of the left side of \eqref{eq:incidence} at $Z$ is the
sum of the four variances. The variance of the right side is
$p^2(n/p)(1-n/p)=n(p-n)$.
\end{proof}
Identity \eqref{eq:incidence} is the case $k=4$ of
\cite[Proposition~3.3]{GKS}, and \eqref{eq:variance} is the four-direction
case of the variance identity in \cite[Section~6]{GKS}.

Let $q_i\in\F_p[t]$ be the projection polynomial of $a_i$, that is, the
unique polynomial of degree at most $p-1$ representing $a_i$ modulo $p$.
After a linear change of coordinates and a rescaling of the forms, we may
assume that
\begin{equation}\label{eq:forms}
 (\ell_1,\ell_2,\ell_3,\ell_4)=(x,y,y-x,y-rx),
 \qquad r\in\F_p\setminus\{0,1\}.
\end{equation}
Indeed, take $\ell_1,\ell_2$ as coordinates. Then $\ell_3$ has two nonzero
coefficients, and rescaling the coordinates turns it into $y-x$. The form
$\ell_4$ also has two nonzero coefficients and is not proportional to
$\ell_3$, so after rescaling it equals $y-rx$ with $r\ne0,1$. Reducing
\eqref{eq:incidence} modulo $p$ gives
\begin{equation}\label{eq:polynomial}
 q_1(x)+q_2(y)+q_3(y-x)+q_4(y-rx)=3n
 \qquad\text{for all }(x,y)\in\F_p^2.
\end{equation}

\begin{lemma}[Degree bound]\label{lem:degree}
Identity \eqref{eq:polynomial} holds in $\F_p[x,y]$, and all four $q_i$
have degree at most two.
\end{lemma}
\begin{proof}
The left side of \eqref{eq:polynomial} minus $3n$ has degree less than $p$
in each variable and vanishes on $\F_p^2$. For each fixed $y_0$, it is a
polynomial in $x$ of degree less than $p$ with $p$ roots, so all its
coefficients vanish at $y_0$. These coefficients are polynomials in $y$ of
degree less than $p$, so they vanish identically. This proves the first
assertion.

Suppose that the largest degree $k$ of the $q_i$ is at least three. The
homogeneous part of degree $k$ of \eqref{eq:polynomial} is then a
nontrivial linear relation among $x^k,y^k,(y-x)^k,(y-rx)^k$. On the
monomials $x^k,y^k,xy^{k-1},x^2y^{k-2}$ their coefficient matrix is
\[
 \begin{pmatrix}
 1&0&(-1)^k&(-r)^k\\
 0&1&1&1\\
 0&0&-k&-kr\\
 0&0&\binom{k}{2}&\binom{k}{2}r^2
 \end{pmatrix},
\]
with determinant $-k\binom{k}{2}r(r-1)$. This is nonzero because
$3\le k<p$ and $r\ne0,1$, a contradiction.
\end{proof}
This is the four-direction case of the degree bound of Kiss and Somlai
\cite[Proposition~3.1]{KissSomlai}; we include the short proof because the
coefficient relations are needed below.

We call a set $R\subseteq\F_p^2$ \emph{reduced} if it contains no full line.
The projection functions of a reduced set take values in $\{0,\ldots,p-1\}$
and are therefore the least residues of their projection polynomials.
Every set with four special directions becomes reduced after deleting the
full lines in one direction (Lemma~\ref{lem:padding}); the proof of that
lemma uses only the degree bound, not the results below.

Write $q_i(t)=A_it^2+B_it+C_i$ with coefficients in $\F_p$.
\begin{lemma}[Genuine quadratics]\label{lem:genuine}
For a reduced set with four special directions, normalized as in
\eqref{eq:forms},
\begin{align}
 (A_1,A_2,A_3,A_4)
  &=\lambda(-r(r-1),r-1,-r,1),\label{eq:A}\\
 (B_1,B_2,B_3,B_4)
  &=(u+rv,-u-v,u,v),\label{eq:B}
\end{align}
for some $\lambda\in\F_p^\times$ and $u,v\in\F_p$. In particular, all four
projection polynomials have degree exactly two.
\end{lemma}
\begin{proof}
Comparing the coefficients of $y^2$, $xy$ and $x^2$ in
\eqref{eq:polynomial} gives $A_2+A_3+A_4=0$, $A_3+rA_4=0$ and
$A_1+A_3+r^2A_4=0$; here we used that $2$ is invertible. Their solution is
\eqref{eq:A} with $\lambda=A_4$. Comparing the coefficients of $x$ and $y$
gives \eqref{eq:B} with $u=B_3$ and $v=B_4$.

Suppose that $\lambda=0$. Since the set is reduced, a constant projection
polynomial would give constant projection counts, which is impossible in a
special direction. Hence all four $q_i$ would be linear and nonconstant,
and each $a_i$ would take every value $0,\ldots,p-1$ exactly once. Then
$n=(p-1)/2$ and each variance equals $(p^2-1)/12$, so
\eqref{eq:variance} would read $(p^2-1)/3=(p^2-1)/4$, which is false.
Thus $\lambda\ne0$, and since $r\ne0,1$, all four $A_i$ are nonzero.
\end{proof}
For reduced sets with four special directions, Ghidelli, Kiss and Somlai
showed that at most one projection polynomial is linear
\cite[Lemma~3.4]{GKS}; Lemma~\ref{lem:genuine} shows that none is.

\begin{proposition}[Common vertex]\label{prop:center}
For a reduced set $R$ with four special directions there is a unique point
$c$ such that each of the four polynomials $s\mapsto q_i(\ell_i(c)+s)$ is
even; equivalently, $\ell_i(c)=-B_i/(2A_i)$ for every $i$. In the notation
of \eqref{eq:A}--\eqref{eq:B},
\begin{equation}\label{eq:center}
 c_x=\frac{u+rv}{2\lambda r(r-1)},\qquad
 c_y=\frac{u+v}{2\lambda(r-1)}.
\end{equation}
Moreover, $R$ is centrally symmetric about $c$.
\end{proposition}
\begin{proof}
The polynomial $A_it^2+B_it+C_i$ is even about $t_0$ exactly when
$t_0=-B_i/(2A_i)$. The conditions for $i=1,2$ give \eqref{eq:center}, and
direct computation gives
\[
 c_y-c_x=\frac{u}{2\lambda r}=-\frac{B_3}{2A_3},\qquad
 c_y-rc_x=-\frac{v}{2\lambda}=-\frac{B_4}{2A_4},
\]
so the conditions for $i=3,4$ also hold. Uniqueness follows because
$\ell_1,\ell_2$ are independent. After translating $c$ to the origin, the
four polynomials become $A_it^2+\delta_i$. Their least residues, which are
the projection functions of $R$, are even functions, and
\eqref{eq:incidence} gives $\one_R(-z)=\one_R(z)$.
\end{proof}

%% file: reflection.tex
\section{Full lines and the reflection defect}\label{sec:reflection}
We first describe the effect of full lines for an arbitrary number of
special directions. Deleting full lines is also the first step in
\cite[Section~3.1]{GKS}; the degree bound shows how many lines can occur.

\begin{lemma}[Bounded padding]\label{lem:padding}
Let $3\le k\le p$ and let $S\subseteq\F_p^2$ have exactly $k$ special
directions. Suppose that $S$ contains a full line, and delete from $S$ all
$t$ full lines parallel to it. The remaining set $R$ is nonempty and
reduced, has the same special directions as $S$, and $t\le k-2$.
Conversely, adding to a nonempty reduced set with $k$ special directions
any collection of its empty lines in one parallel class preserves the
special directions.
\end{lemma}
\begin{proof}
Every line transverse to the deleted lines meets one of them, so it is not
full in $R$; no parallel full line remains by construction. If $R$ were
empty, $S$ would be a nonempty proper union of parallel lines, with only one
special direction. Hence $R$ is nonempty and reduced.

The count on every transverse line decreases by exactly $t$, so the
transverse directions keep their type. The direction of the deleted lines
was special for $S$, since equidistribution together with a full line would
force $S=\F_p^2$. It remains special for $R$, since equidistribution
together with an empty line would force $R=\emptyset$. The projection
polynomial of $R$ in this direction is therefore nonconstant, has degree at
most $k-2$ by \cite[Proposition~3.1]{KissSomlai}, and vanishes at the $t$
intercepts of the deleted lines. Hence $t\le k-2$.

For the converse, adding lines changes every transverse count by the same
amount. If at least one line is added, then in its parallel class it has
$p$ points, whereas any line of that class meeting the nonempty reduced set
has between $1$ and $p-1$ points; so this direction remains special.
\end{proof}

Now let $k=4$, and let $R$ be obtained from $S$ by deleting all full lines
in one direction. Deletion changes every transverse projection function by
a constant. In the deletion direction it changes some counts from $p$ to
$0$, which leaves the projection polynomial unchanged. Hence $S$ and $R$
have projection polynomials with the same quadratic and linear
coefficients. By Lemma~\ref{lem:genuine} and Proposition~\ref{prop:center},
the four projection polynomials of $S$ have degree exactly two and have a
unique common vertex point $c$. This point is determined by the
projection polynomials of $S$ itself, so it does not depend on the
direction chosen for the deletion.

The vertex condition is independent of the choice of defining forms. If
$\ell=\alpha\ell_i$ with $\alpha\in\F_p^\times$, then
$q_\ell(t)=q_i(t/\alpha)$, and $q_\ell(\ell(c)+s)=q_i(\ell_i(c)+s/\alpha)$
is even in $s$. If $\ell$ defines a nonspecial direction, then $q_\ell$ is
constant. Thus $s\mapsto q_\ell(\ell(c)+s)$ is even for every nonzero linear
form $\ell$, and $c$ is the unique point with this property. This proves the
first assertion of Theorem~\ref{thm:reflection}.

\begin{proof}[Completion of the proof of Theorem~\ref{thm:reflection}]
Translate $c$ to the origin. If $S$ is reduced,
Proposition~\ref{prop:center} shows that $S$ is centrally symmetric.
Otherwise, by Lemma~\ref{lem:padding}, $S$ is the union of a nonempty
reduced set $R$ and some empty lines of $R$ in one parallel class, say with
defining form $\ell_i$. The empty lines of $R$ in this class are the lines
$\ell_i=t$ with $A_it^2+\delta_i=0$. The set of such $t$ is therefore
empty, $\{0\}$, or $\{a,-a\}$ with $a\ne0$. By
Proposition~\ref{prop:center}, $R=-R$.

If the set of added intercepts is invariant under $t\mapsto-t$, then $S=-S$.
Otherwise exactly one of the lines $\ell_i=a$ and $\ell_i=-a$ was added; call
it $L$, and put $L'=-L$. Since $\one_S=\one_R+\one_L$ and
$\one_S(-z)=\one_R(z)+\one_{L'}(z)$, we obtain
$\one_S(z)-\one_S(-z)=\one_L(z)-\one_{L'}(z)$, which is
\eqref{eq:reflection-main}.

The set $R$ misses both $L$ and $L'$, so a line transverse to them contains
at most $p-2$ points of $R$ and exactly one point of $L$. Its count in $S$
therefore lies between $1$ and $p-1$. In the class of $L$, the only empty
lines of $R$ are $L$ and $L'$, and only $L$ was added. Hence $L$ is the only
full line of $S$, and $L'$ is the only empty line. The sets $S\setminus L=R$
and $S\cup L'=R\cup L\cup L'$ are centrally symmetric, and by
Lemma~\ref{lem:padding} they have the same special directions as $S$.
In both cases at most two lines were added to $R$.

Finally, if $S$ is centrally symmetric about a point $b$, then reflection
about $b$ maps the line $\ell_i=t$ to the line $\ell_i=2\ell_i(b)-t$, so
$q_i(t)=q_i(2\ell_i(b)-t)$. For a polynomial of degree two this forces
$\ell_i(b)$ to be the vertex of $q_i$, and two independent forms give $b=c$.
Hence in the second case $S$ has no center of symmetry, and the two cases
of the theorem are mutually exclusive.
\end{proof}

\begin{corollary}\label{cor:center-equivalence}
For subsets of $\F_p^2$ with exactly four special directions and quadratic
center at the origin, affine equivalence coincides with
$\GL(2,p)$-equivalence, and the stabilizers in $\AGL(2,p)$ and in
$\GL(2,p)$ coincide.
\end{corollary}
\begin{proof}
Let $g(z)=Mz+w$ be an affine bijection. For a linear form $\ell$, the
projection function of $g(S)$ along $\ell\circ M^{-1}$ is the projection
function of $S$ along $\ell$ composed with the translation
$t\mapsto t-\ell(M^{-1}w)$. The vertex condition of
Theorem~\ref{thm:reflection} therefore shows that the quadratic center of
$g(S)$ is $g(c)$. If both centers are at the origin, then $w=0$. The same
argument applies to stabilizers.
\end{proof}

%% file: feasibility.tex
\section{An exact profile reconstruction criterion}\label{sec:feasibility}
For $b\in\F_p$ let $[b]_p\in\{0,\ldots,p-1\}$ denote its least nonnegative
residue, and let $\chi$ be the quadratic character of $\F_p$, with
$\chi(0)=0$. Fix $r\ne0,1$ and $\lambda\ne0$, and define $A_1,\ldots,A_4$ by
\eqref{eq:A}. For $\delta_i\in\F_p$ and a subset
\[
 Z_i\subseteq\{t\in\F_p:A_it^2+\delta_i=0\},
\]
define the integer-valued profiles
\begin{equation}\label{eq:lifts}
 a_i(t)=[A_it^2+\delta_i]_p+p\one_{Z_i}(t).
\end{equation}
A count of $0$ or $p$ has residue $0$, so $Z_i$ records which zeros of the
residue correspond to full lines. By Theorem~\ref{thm:reflection} and the
normalization \eqref{eq:forms}, every set with four special directions and
quadratic center at the origin has projection functions of this form.

\begin{proposition}[Integer moment criterion]\label{prop:criterion}
There is a set $S\subseteq\F_p^2$ whose special directions are exactly
the four directions defined by the forms \eqref{eq:forms}, and whose
projection functions along these forms are the profiles \eqref{eq:lifts},
if and only if, for some integer $1\le n<p$,
\begin{align}
 \sum_{t\in\F_p}a_i(t)&=pn\quad(1\le i\le4),\label{eq:means}\\
 \sum_{i=1}^4\delta_i&=3n\quad\text{in }\F_p,\label{eq:constant}\\
 \sum_{i=1}^4\sum_{t\in\F_p}a_i(t)^2&=pn(p+3n).\label{eq:squares}
\end{align}
In that case $S$ is unique, has $np$ points, and is given by
\begin{equation}\label{eq:reconstruct}
 S=\left\{z\in\F_p^2:\sum_{i=1}^4a_i(\ell_i(z))=3n+p\right\}.
\end{equation}
\end{proposition}
\begin{proof}
Suppose that $S$ exists. Then \eqref{eq:means} holds with $|S|=np$,
\eqref{eq:constant} follows by comparing constant terms in
\eqref{eq:polynomial}, and \eqref{eq:reconstruct} is a restatement of
\eqref{eq:incidence}. Since every $a_i(U)$ has mean $n$, the variance
identity \eqref{eq:variance} reads
$p^{-1}\sum_{i,t}a_i(t)^2-4n^2=n(p-n)$, which is \eqref{eq:squares}.

Conversely, assume \eqref{eq:means}--\eqref{eq:squares}. Put
$W(z)=\sum_ia_i(\ell_i(z))$ and $Y(z)=(W(z)-3n)/p$. Modulo $p$, $W(z)$
equals $\sum_iA_i\ell_i(z)^2+\sum_i\delta_i$. The quadratic form
$\sum_iA_i\ell_i^2$ vanishes identically by the choice \eqref{eq:A}, so
\eqref{eq:constant} shows that $Y$ is integer-valued. For $Z$ uniform on
$\F_p^2$, pairwise independence of the $\ell_i(Z)$ and
\eqref{eq:means}--\eqref{eq:squares} give
\[
 \E W=4n,\qquad \Var(W)=n(p-n),\qquad
 \E Y=\frac np,\qquad \E Y^2=\frac np.
\]
Hence $\E[Y(Y-1)]=0$. Since $y(y-1)\ge0$ for every integer $y$, it follows
that $Y(z)\in\{0,1\}$ for all $z$. Let $S$ be the set where $Y=1$; then
$|S|=p^2\E Y=np$.

Fix $j$ and $t$, and sum $pY=W-3n$ over the line $\ell_j=t$. The $j$th term
of $W$ contributes $pa_j(t)$. Every other form is a bijection from this line
onto $\F_p$, so each of the three other terms contributes $pn$. Hence the
line contains exactly $a_j(t)$ points of $S$. On a line in any other
direction, all four forms are bijective, so the line contains exactly $n$
points. Finally, $a_j$ is not constant, because its reduction modulo $p$ is
the polynomial $A_jt^2+\delta_j$ of degree $2<p$. Thus the four given
directions are special and the others are not. Uniqueness follows from
\eqref{eq:incidence}.
\end{proof}

Every set with four special directions is equivalent, by a linear map
followed by a translation of its quadratic center to the origin, to a set
of the form in Proposition~\ref{prop:criterion}. Conversely, every
parameter tuple satisfying the criterion gives such a set. All such sets
therefore arise from the finite parameter domain
\begin{equation}\label{eq:domain}
 r=2,\ldots,p-1,\quad \lambda=1,\ldots,p-1,\quad
 \delta_i=0,\ldots,p-1,\quad
 Z_i\subseteq\{t:A_it^2+\delta_i=0\},
\end{equation}
where field elements are represented by least residues. The conditions
\eqref{eq:means}--\eqref{eq:squares} are integer tests that can be applied
before any point set is constructed.

For reduced sets all $Z_i$ are empty. For fixed $A\ne0$, the number of
$t\in\F_p$ with $[At^2+\delta]_p=j$ is
\begin{equation}\label{eq:quadratic-law}
 1+\chi(A)\chi(j-\delta)\qquad(0\le j\le p-1),
\end{equation}
so the value multiset of $[At^2+\delta]_p$ depends only on $\delta$ and
$\chi(A)$. Moment tests therefore require only $2p$ one-variable value
tables, although the exact reconstruction must still respect the coupling
of the four coefficients through $r$ and $\lambda$.

%% file: threshold.tex
\section{The exact prime threshold}\label{sec:threshold}
We now determine the largest prime admitting a set with four special
directions. Following the second-moment strategy of \cite[Section~6]{GKS},
we compare the variance identity \eqref{eq:variance} with a lower bound for
the variance of each projection function. Since such a bound concerns
least residues rather than arbitrary lifts from $0$ to $p$, we first delete
full lines.

\subsection{A uniform variance bound}
Let $m=(p-1)/2$ and $H_m=\sum_{j=1}^m1/j$.

\begin{lemma}\label{lem:profile-variance}
If $A\ne0$ and $X=[AU^2+\delta]_p$ with $U$ uniform on $\F_p$, then
\begin{equation}\label{eq:var-bound}
 \Var(X)\ge
 \frac{p^2-1}{12}-\frac{p^2-1}{6\sqrt p}
                  -\frac{pH_m^2}{4}.
\end{equation}
\end{lemma}
\begin{proof}
By \eqref{eq:quadratic-law}, the distribution of $X$ differs from the
uniform distribution on $\{0,\ldots,p-1\}$ by $\chi(A)\chi(j-\delta)/p$ at
$j$. Hence, for a real function $g$ on $\{0,\ldots,p-1\}$, identified with
$\F_p$, the expectations of $g$ under the two distributions differ in
absolute value by
\begin{equation}\label{eq:correlation}
 \frac1p\left|\sum_j g(j)\chi(j-\delta)\right|
 \le \frac{\sqrt p}{p^2}\sum_{h=1}^{p-1}|\widehat g(h)|,
\end{equation}
where $\widehat g$ is the unnormalized Fourier transform of
\eqref{eq:fourier} with $d=1$. The inequality follows from Fourier
inversion, since $\sum_j\chi(j-\delta)=0$ and every Gauss sum
$\tau_h=\sum_j\chi(j)e_p(hj)$ with $h\ne0$ has modulus $\sqrt p$. For
completeness, substituting $j=tv$ with $v\ne0$ gives
\[
 |\tau_h|^2=\sum_{j,v\in\F_p}\chi(j)\chi(v)e_p(h(j-v))
 =\sum_{t}\chi(t)\sum_{v\ne0}e_p(hv(t-1))
 =(p-1)-\sum_{t\ne1}\chi(t)=p.
\]

Put $b=(p-1)/2$. For a $p$th root of unity $z\ne1$, summing the geometric
series and differentiating before substituting $z^p=1$ gives
\begin{align}
 \sum_{j=0}^{p-1}jz^j&=-\frac{p}{1-z},\label{eq:dft-linear}\\
 \sum_{j=0}^{p-1}(j-b)^2z^j
     &=-\frac{p(1+z)}{(1-z)^2}.\label{eq:dft-square}
\end{align}
For $z=e_p(h)$ these have absolute values at most
$p/(2|\sin(\pi h/p)|)$ and $p/(2\sin^2(\pi h/p))$, respectively. We also
use
\begin{equation}\label{eq:trig-sums}
 \sum_{h=1}^{p-1}\frac1{|\sin(\pi h/p)|}\le pH_m,
 \qquad
 \sum_{h=1}^{p-1}\frac1{\sin^2(\pi h/p)}
       =\frac{p^2-1}{3}.
\end{equation}
The first follows by pairing $h$ with $p-h$ and using
$\sin(\pi h/p)\ge2h/p$ for $1\le h\le m$. The second follows by
differentiating $\sum_{h=0}^{p-1}\cot(x+\pi h/p)=p\cot(px)$ and letting
$x\to0$ after cancelling the singular terms.

Applying \eqref{eq:correlation} to $g(j)=j-b$ and $g(j)=(j-b)^2$, whose
averages under the uniform distribution are $0$ and $(p^2-1)/12$, yields
\[
 |\E X-b|\le\frac{\sqrt p\,H_m}{2},\qquad
 \left|\E(X-b)^2-\frac{p^2-1}{12}\right|
       \le\frac{p^2-1}{6\sqrt p}.
\]
Since $\Var(X)=\E(X-b)^2-(\E X-b)^2$, this proves \eqref{eq:var-bound}.
\end{proof}

\begin{proposition}\label{prop:analytic}
No prime $p\ge2048$ admits a subset of $\F_p^2$ with exactly four special
directions.
\end{proposition}
\begin{proof}
By Lemma~\ref{lem:padding} we may assume that the set is reduced, and by
Lemma~\ref{lem:genuine} and Proposition~\ref{prop:center} each of its four
projection functions is, up to a translation of the variable, of the form
$[At^2+\delta]_p$ with $A\ne0$. Four times the right side of
\eqref{eq:var-bound} exceeds $p^2/4$ provided
\begin{equation}\label{eq:cutoff-test}
 \frac1{12}-\frac1{3p^2}-\frac2{3\sqrt p}
               -\frac{H_m^2}{p}>0.
\end{equation}
To verify \eqref{eq:cutoff-test}, note that $H_m\le1+\log p$ and that
$(1+\log x)^2/x$ is decreasing for $x>e$. Hence
\[
 \frac{H_m^2}{p}
 \le\frac{(1+\log2048)^2}{2048}
 <\frac{81}{2048}<\frac1{24}.
\]
Also $2/(3\sqrt p)<1/48$, so the left side of \eqref{eq:cutoff-test} is
greater than $1/48-1/(3p^2)>0$. On the other hand, \eqref{eq:variance}
requires the four variances to sum to $n(p-n)\le p^2/4$, a contradiction.
\end{proof}

For four special directions, the second-moment theorem of \cite{GKS} applies
from $p=207769$ on \cite[Remark following Theorem~1.3]{GKS}. That theorem
treats projection polynomials of arbitrary bounded degree through
discrepancy estimates. Here each reduced projection polynomial is a genuine
quadratic, so completing its square gives a projection function of the form
$[At^2+\delta]_p$, up to a translation of the variable. Its distribution is given
exactly by \eqref{eq:quadratic-law}; this is what lowers the cutoff to
$2048$.

\subsection{The remaining finite interval}
\begin{proposition}[Exact finite verification]\label{prop:finite}
For none of the $302$ primes $19\le p<2048$ does $\F_p^2$ contain a subset
with exactly four special directions.
\end{proposition}
\begin{proof}[Computer-assisted proof]
By Lemma~\ref{lem:padding} it suffices to exclude reduced sets.

For $p=19,23,29,31,37$ we enumerate the reduced part of the domain
\eqref{eq:domain}: all $r$, $\lambda$, and all quadruples of constants whose
four profiles have a common integer mean. We discard the quadruples
violating \eqref{eq:constant} and test each remaining one pointwise for
$W(z)\in\{3n,3n+p\}$. No quadruple passes; the counts are listed in
Table~\ref{tab:small-exclusion}. All tests use integer arithmetic.

For each of the $297$ primes $41\le p<2048$ a necessary moment test
suffices. Fix a nonsquare $\nu\in\F_p^\times$. For every integer $n$ that
occurs as the mean of one of the $2p$ profiles $[At^2+\delta]_p$ with
$A\in\{1,\nu\}$ and $\delta\in\F_p$, put
\begin{equation}\label{eq:V}
 V_p(n)=\min\left\{\sum_{t=0}^{p-1}
       ([At^2+\delta]_p-n)^2:\ \sum_t[At^2+\delta]_p=pn\right\}.
\end{equation}
By \eqref{eq:quadratic-law}, this is also the minimum over all
$A\in\F_p^\times$. If the four reduced profiles of a set have mean $n$,
their variances sum to at least $4V_p(n)/p$. An exact integer computation
shows that
\begin{equation}\label{eq:finite-gap}
 4V_p(n)-pn(p-n)>0
\end{equation}
for every such $n$ and each of these primes, the largest being $2039$.
Hence \eqref{eq:variance} fails.

Both finite procedures are specified completely in \ref{app:algorithms}.
The variance test was also implemented independently, using the value
multiplicities \eqref{eq:quadratic-law} and an integer moment recurrence;
the two tables of margins agree entry by entry, and the smallest margin,
$492$, occurs at $p=41$. For $p\le37$, the direct reconstruction agrees with
an independent enumeration based on the moment criterion.
\end{proof}

\begin{table}[ht]
\centering
\caption{Reduced profile tuples in the direct exclusion for small primes.
The common-mean column counts the tuples whose four profiles have a common integer
mean; the congruent column counts those that also satisfy \eqref{eq:constant}.}
\label{tab:small-exclusion}
\begin{tabular}{rrrr}
\toprule
$p$ & Common-mean tuples & Congruent tuples & Boolean sets\\
\midrule
19&491130&19161&0\\
23&1837902&81708&0\\
29&6051780&240576&0\\
31&4817190&128130&0\\
37&15005340&369504&0\\
\bottomrule
\end{tabular}
\end{table}

\subsection{An example over \texorpdfstring{$\F_{17}$}{F17}}
\begin{proposition}\label{prop:17}
In $\F_{17}^2$, let
\begin{equation}\label{eq:17W}
 W(x,y)=[-2x^2+6]_{17}+[y^2+6]_{17}
       +[-2(y-x)^2+6]_{17}+[(y-2x)^2+6]_{17}.
\end{equation}
Then $S_{17}=\{(x,y):W(x,y)=41\}$ has $136$ points and exactly four special
directions.
\end{proposition}
\begin{proof}
The coefficients are those of \eqref{eq:A} with $r=2$ and $\lambda=1$.
Since $1$ and $-2\equiv15$ are squares modulo $17$, the four profiles have
the same value multiset. The values of $[t^2+6]_{17}$ for
$t=0,\ldots,16$ are
\[
 6,7,10,15,5,14,8,4,2,2,4,8,14,5,15,10,7.
\]
Their sum is $136$, and the sum of their squared deviations from $8$ is
$306$; so each profile has mean $8$ and variance $18$. The constants sum to
$24=3\cdot8$, and $4\cdot18=8(17-8)$. Hence
Proposition~\ref{prop:criterion} applies with $n=8$: the function $W$ takes
only the values $24$ and $41$, and $S_{17}$ has $136$ points and exactly
four special directions.
\end{proof}
Propositions~\ref{prop:analytic}, \ref{prop:finite} and~\ref{prop:17} prove
that $T_4=17$. The next section determines all configurations at the five
remaining primes.

%% file: classification.tex
\section{The complete affine list}\label{sec:classification}
It remains to classify the sets over $\F_p$ for $p=5,7,11,13,17$.
Each row of Table~\ref{tab:representatives} specifies parameters
$p,n,r,\lambda,(\delta_1,\delta_2,\delta_3,\delta_4)$ and lift sets $Z_i$
for Proposition~\ref{prop:criterion}. To obtain the set, compute
$A_1,\ldots,A_4$ from $(r,\lambda)$ by \eqref{eq:A}, form the profiles
\eqref{eq:lifts} along $x,y,y-x,y-rx$, and apply \eqref{eq:reconstruct}.
For each $i$, the set $Z_i$ consists precisely of the intercepts $t$ listed
as $i:t$ in the column of lifts; a dash means that all four sets are empty. The last column gives the
order of the stabilizer in $\AGL(2,p)$, which equals the stabilizer in
$\GL(2,p)$ by Corollary~\ref{cor:center-equivalence}. Each row has fewer
than $p^2/2$ points, and the other $27$ representatives are the complements
of the rows. Complementation preserves the special directions, the
stabilizer and the quadratic center, since it replaces each projection
polynomial $q$ by $-q$.

\begin{table}[htp]
\centering
\caption{Representatives of all $54$ classes: the $27$ sets below and
their complements, reconstructed by
\eqref{eq:lifts}--\eqref{eq:reconstruct}.}
\label{tab:representatives}
\begin{tabular}{rrrrrlcr}
\toprule
Row&$p$&$n$&$r$&$\lambda$&$(\delta_1,\delta_2,\delta_3,\delta_4)$&
Lifts&Stab.\\
\midrule
1 & 5 & 1 & 4 & 4 & $(3,3,1,1)$ & -- & 8 \\
2 & 5 & 2 & 4 & 4 & $(3,4,2,2)$ & $1:1$ & 2 \\
3 & 5 & 2 & 2 & 3 & $(1,4,2,4)$ & $1:1$ & 2 \\
4 & 5 & 2 & 3 & 2 & $(2,2,2,0)$ & -- & 4 \\
5 & 5 & 2 & 4 & 2 & $(4,2,0,0)$ & -- & 4 \\
6 & 5 & 2 & 3 & 1 & $(2,4,0,0)$ & -- & 2 \\
7 & 7 & 2 & 6 & 1 & $(2,4,0,0)$ & -- & 4 \\
8 & 7 & 3 & 6 & 1 & $(2,5,1,1)$ & $1:1$ & 2 \\
9 & 7 & 3 & 4 & 5 & $(4,1,1,3)$ & $1:1$ & 2 \\
10 & 7 & 3 & 3 & 6 & $(1,5,5,5)$ & -- & 6 \\
11 & 7 & 3 & 2 & 5 & $(3,5,3,5)$ & -- & 8 \\
12 & 7 & 3 & 4 & 1 & $(5,5,5,1)$ & -- & 4 \\
13 & 7 & 2 & 5 & 1 & $(6,0,0,0)$ & -- & 6 \\
14 & 7 & 3 & 5 & 1 & $(6,1,1,1)$ & $1:1$ & 3 \\
15 & 7 & 3 & 4 & 5 & $(5,0,1,3)$ & $2:0$ & 2 \\
16 & 7 & 3 & 3 & 5 & $(3,5,5,3)$ & -- & 4 \\
17 & 7 & 3 & 5 & 1 & $(0,0,1,1)$ & $1:0;\,2:0$ & 4 \\
18 & 11 & 3 & 2 & 8 & $(5,5,5,5)$ & -- & 16 \\
19 & 11 & 4 & 2 & 8 & $(5,6,6,6)$ & $1:1$ & 2 \\
20 & 11 & 5 & 10 & 2 & $(5,7,7,7)$ & $1:2;\,1:9$ & 4 \\
21 & 11 & 4 & 8 & 5 & $(6,6,0,0)$ & -- & 4 \\
22 & 11 & 5 & 4 & 8 & $(9,7,3,7)$ & -- & 2 \\
23 & 11 & 5 & 9 & 9 & $(0,7,7,1)$ & $1:0$ & 2 \\
24 & 13 & 5 & 12 & 4 & $(9,9,5,5)$ & -- & 8 \\
25 & 13 & 5 & 5 & 8 & $(5,9,5,9)$ & -- & 4 \\
26 & 13 & 6 & 4 & 9 & $(6,6,0,6)$ & -- & 6 \\
27 & 17 & 8 & 2 & 2 & $(6,6,6,6)$ & -- & 16 \\
\bottomrule
\end{tabular}
\end{table}

\begin{proof}[Computer-assisted proof of Theorem~\ref{thm:classification}]
By Section~\ref{sec:threshold}, only the five primes in
Table~\ref{tab:spectrum} remain. Fix one of them. By the discussion after
Proposition~\ref{prop:criterion}, every set with four special directions is
equivalent under $\AGL(2,p)$ to a \emph{normalized} set, that is, one
obtained from a parameter tuple in \eqref{eq:domain} that satisfies the
criterion; such a set has its special directions defined by
$x,y,y-x,y-rx$ for some $r$ and its quadratic center at the origin. We
enumerate the whole domain \eqref{eq:domain}, including every choice of
lift sets, apply Proposition~\ref{prop:criterion}, reconstruct every
surviving tuple, and remove repeated point sets. The resulting numbers of
normalized sets are listed in the column ``Full candidates'' of
Table~\ref{tab:candidates}.

For each of the $54$ representatives we apply every matrix of $\GL(2,p)$.
The resulting orbits are pairwise disjoint, so the representatives are
pairwise inequivalent. The normalized sets in the union of these orbits are
exactly the enumerated normalized sets, so every set with four special
directions is equivalent to a representative. The number of matrices fixing
each representative gives the stabilizer column of
Table~\ref{tab:representatives}, and every orbit satisfies
\[
 |\operatorname{Orb}_{\GL(2,p)}(S)|\,
 |\operatorname{Stab}_{\GL(2,p)}(S)|
       =(p^2-1)(p^2-p).
\]
By Corollary~\ref{cor:center-equivalence}, this linear classification of
sets centered at the origin is an affine classification of all sets,
including those of the second type in Theorem~\ref{thm:reflection}.

The computation was repeated by a second method: we enumerated the reduced
sets, added every admissible collection of empty parallel lines as in
Lemma~\ref{lem:padding}, and classified the results using normalizations of
triples of special directions and, separately, generators of $\GL(2,p)$.
This gives the same classes and stabilizers. As a check independent of the
quadratic parametrization, all subsets of $\F_5^2$ with $5$ and $10$ points
were tested directly; complementation covers the sizes $15$ and $20$.
The exhaustive enumeration and the orbit coverage tests are specified in
\ref{app:algorithms}.
\end{proof}

\begin{table}[ht]
\centering
\caption{Numbers of distinct normalized sets, before linear equivalence.
Reduced candidates contain no full line; full candidates allow all lift
sets and include the reduced candidates. The last two columns count the
corresponding affine classes.}
\label{tab:candidates}
\begin{tabular}{rrrrr}
\toprule
$p$&Reduced candidates&Full candidates&Reduced classes&All classes\\
\midrule
5&108&408&4&12\\
7&318&948&9&22\\
11&495&990&6&12\\
13&264&312&5&6\\
17&48&48&2&2\\
\midrule
Total&1233&2706&26&54\\
\bottomrule
\end{tabular}
\end{table}

\begin{proof}[Proof of Corollary~\ref{cor:T4}]
The equality $T_4=17$ was proved at the end of
Section~\ref{sec:threshold}. Over $\F_{13}$, rows $24$ and $25$ of
Table~\ref{tab:representatives} are the only classes of size $65$, and all
other classes are larger; their stabilizers have orders $8$ and $4$.
\ref{app:13} gives a direct proof that there is no example of the next
smaller admissible size, $52$.
\end{proof}

\begin{corollary}\label{cor:cylinders}
Let $d\ge2$ and let $p\ge5$ be prime. Every subset of $\F_p^d$ with exactly
four special parallel classes of hyperplanes is the full preimage of a set
specified in Table~\ref{tab:representatives} or its complement under a surjective affine map
$\F_p^d\to\F_p^2$, and taking this quotient induces a bijection between the
affine classes of such subsets of $\F_p^d$ and the affine classes of
Theorem~\ref{thm:classification} at $p$. In particular, a planar class of
sets of size $np$ corresponds to a class of sets of size $np^{d-1}$.
\end{corollary}
\begin{proof}
By Section~\ref{sec:fourier}, the special parallel classes of hyperplanes
correspond to the active lines of the indicator function. Apply
Theorem~\ref{thm:dimension}: the quotient function has exactly four special
directions, and two sets are affinely equivalent exactly when their
quotients are. Every point of $\F_p^d/K$ has $p^{d-2}$ preimages.
\end{proof}

\section{Concluding remarks}\label{sec:conclusion}
The projection quadratics explain why these sets admit a complete
classification. Their common vertex is an intrinsic point even when the set
has no center of symmetry, and in that case reflection about it exchanges
one full line with one empty line and preserves the rest of the set setwise.
Translating this point to the origin reduces affine equivalence to linear
equivalence, and the moment criterion determines exactly which quadratic
profiles come from sets. The same description gives the sharp threshold
$T_4=17$ and, over $\F_{13}$, the two minimal classes of size $65$.
The Fourier dimension theorem shows that the planar problem suffices: over
an odd prime field, a higher-dimensional set with exactly four special
hyperplane classes is a cylinder over a planar one.

We mention some questions suggested by this work.
\begin{enumerate}
\item The classification and the exclusion of the primes $19\le p<2048$
rely on exact computation. It would be interesting to find a
computer-free proof that $T_4=17$, or a conceptual explanation of the
exceptional set over $\F_{17}$, whose four profiles all coincide with
$t\mapsto[t^2+6]_{17}$ after rescaling the variable and whose stabilizer
has order $16$.
\item For five special directions the projection polynomials have degree
at most three \cite[Proposition~3.1]{KissSomlai}. Does the common vertex of
Proposition~\ref{prop:center} have an analogue in that case? More generally,
Ghidelli, Kiss and Somlai ask for the exact values of $T_k$ for $k\ge5$
\cite[Question~2]{GKS}.
\item Sets with three special directions have been classified in all
finite translation planes \cite{ACW}. It would be natural to study sets
with four special directions in $\mathrm{AG}(2,q)$ when $q$ is a proper
prime power.
\end{enumerate}

%% file: obstruction13.tex
\section{The direct obstruction at fifty-two points}\label{app:13}
Theorem~\ref{thm:classification} implies that no $52$-point subset of
$\F_{13}^2$ has exactly four special directions. This case, asked about by
Kiss and Somlai \cite[Section~6]{KissSomlai}, also has a short direct proof
that uses Sections~\ref{sec:quadratics}--\ref{sec:feasibility} but no
computer search.

Suppose that $S\subseteq\F_{13}^2$ has $52$ points and exactly four special
directions, so $n=4$. Normalize as in \eqref{eq:forms} and translate the
quadratic center to the origin. By Section~\ref{sec:reflection}, the
projection polynomials of $S$ are $A_it^2+\delta_i$ with $A_i$ as in
\eqref{eq:A}; in particular, none of them is constant or linear. Let
\[
 Q_{+}=\{1,3,4,9,10,12\},\qquad
 Q_{-}=\{2,5,6,7,8,11\}
\]
be the sets of nonzero squares and nonsquares modulo $13$. If
$\chi(A)=\varepsilon\in\{1,-1\}$, then $At^2$ takes the value $0$ once and
each value in $Q_\varepsilon$ twice, so
\[
 \sum_{t\in\F_{13}}[At^2+\delta]_{13}=T_\varepsilon(\delta),\qquad
 T_\varepsilon(\delta)=[\delta]_{13}
          +2\sum_{v\in Q_\varepsilon}[\delta+v]_{13}.
\]
The values are as follows:
\[
\setlength{\arraycolsep}{3.5pt}
\begin{array}{c|rrrrrrrrrrrrr}
\delta&0&1&2&3&4&5&6&7&8&9&10&11&12\\ \hline
T_+&78&65&78&65&52&65&78&91&104&91&78&91&78\\
T_-&78&91&78&91&104&91&78&65&52&65&78&65&78
\end{array}
\]
By \eqref{eq:lifts}, each projection function of $S$ sums to
$T_{\chi(A_i)}(\delta_i)+13|Z_i|=52$. Since every entry of the table is at
least $52$, we get $Z_i=\emptyset$ and $T_{\chi(A_i)}(\delta_i)=52$; thus
$S$ has no full line and
\begin{equation}\label{eq:13delta}
 \delta_i=6-2\chi(A_i)\qquad(1\le i\le4).
\end{equation}

Condition \eqref{eq:constant} requires $\sum_i\delta_i=3n=12$ in
$\F_{13}$. By \eqref{eq:A} and \eqref{eq:13delta}, this becomes
\[
 24-2\chi(\lambda)\Sigma\equiv12\pmod{13},\qquad
 \Sigma=\chi(-r(r-1))+\chi(r-1)+\chi(-r)+1.
\]
By multiplicativity,
\[
 \Sigma=(\chi(-r)+1)(\chi(r-1)+1)\in\{0,4\},
\]
so the congruence would require $\chi(\lambda)\Sigma\equiv6\pmod{13}$ with
$\chi(\lambda)\Sigma\in\{0,4,-4\}$, which is impossible. Hence no such set
exists.

The minimum size over $\F_{13}$ can also be obtained without the
classification. For sets whose size is divisible by $p$, Ghidelli's notion
of special direction agrees with ours \cite[Section~1]{KissSomlai}. By
\cite[Theorem~1.3]{Ghidelli}, a set of size $13n$ that is not contained in
a union of $n$ lines has at least $\lceil(13+n+2)/(n+1)\rceil$ special
directions, which equals $8,6,5$ for $n=1,2,3$. A set of size $13n$
contained in a union of $n$ lines is a union of $n$ parallel lines, and has
only one special direction, and the empty set has none. Together with the
obstruction above, this excludes all sizes below $65$, and the $65$-point set of Kiss and Somlai
\cite[Section~6]{KissSomlai} shows that $65$ is attained.

%% file: algorithms.tex
\section{Exact finite verification procedures}\label{app:algorithms}
All tests described here represent field elements by least residues and
use only integer arithmetic and exact comparisons; no decision depends on
floating-point computation. We state the loop domains explicitly, so that
each finite assertion is a well-defined finite check.

\subsection{Profile enumeration}
The following Python reference procedure enumerates the normalized sets of
Section~\ref{sec:classification} at a fixed prime, each encoded by its set
of coordinate pairs. The option \texttt{lifts=False} gives the reduced sets
and \texttt{lifts=True} gives all sets. The one-variable profiles are
tabulated once, but the four leading coefficients remain coupled through
$r$ and $\lambda$. At each of the five primes $5,7,11,13,17$, the output
coincides with that of an independent C implementation, as collections of
point sets and not merely in the counts of Table~\ref{tab:candidates}.
The ancillary files contain this procedure, the C implementation, the
representative data, and a driver for the profile, orbit and
finite-exclusion checks.

\begin{small}
\begin{verbatim}
from itertools import product

def profiles(p, A, lifts):
    ans = {n: [] for n in range(1, p)}
    for delta in range(p):
        base = [(A*t*t + delta) % p for t in range(p)]
        roots = [t for t in range(p) if base[t] == 0]
        for bits in range(1 << len(roots) if lifts else 1):
            a = base[:]
            for j, t in enumerate(roots):
                if (bits >> j) & 1:
                    a[t] = p
            total = sum(a)
            if total % p:
                continue
            n = total // p
            if 0 < n < p:
                ans[n].append((delta, a, sum(v*v for v in a)))
    return ans

def candidates(p, lifts):
    tables = {A: profiles(p, A, lifts) for A in range(1, p)}
    result = set()
    for r in range(2, p):
        for lam in range(1, p):
            A = [(-lam*r*(r-1)) % p, (lam*(r-1)) % p,
                 (-lam*r) % p, lam]
            for n in range(1, p):
                for rec in product(*(tables[a][n] for a in A)):
                    if (sum(v[0] for v in rec) - 3*n) % p:
                        continue
                    if sum(v[2] for v in rec) != p*n*(p+3*n):
                        continue
                    a = [v[1] for v in rec]
                    S = set()
                    for x, y in product(range(p), repeat=2):
                        w = (a[0][x] + a[1][y] +
                             a[2][(y-x) % p] +
                             a[3][(y-r*x) % p])
                        assert w in (3*n, 3*n+p)
                        if w == 3*n+p:
                            S.add((x, y))
                    assert len(S) == p*n
                    result.add(frozenset(S))
    return result
\end{verbatim}
\end{small}
For the direct pointwise exclusion of Table~\ref{tab:small-exclusion},
omit the test of the sum of squares and instead reject a tuple at the first
point where $w\notin\{3n,3n+p\}$; count the tuples before and after the
test of \eqref{eq:constant}. This check does not use the sufficiency part
of Proposition~\ref{prop:criterion}. In addition, every reconstructed set
is checked by counting its points on all lines of all $p+1$ parallel
classes.

\subsection{Finite variance exclusion}
For each integer $p$ from $41$ to $2047$, test primality by trial division
up to $\lfloor\sqrt p\rfloor$. For each prime $p$, find the smallest
positive nonsquare $\nu$ by Euler's criterion and initialize all minima to
infinity. For $A=1,\nu$ and $\delta=0,\ldots,p-1$, compute the integer
vector $b_t=[At^2+\delta]_p$, its sum $s$, and $n=s/p$; the sum is
divisible by $p$ because $\sum_tt^2=0$ in $\F_p$ for $p\ge5$. Update the
minimum at $n$ with $\sum_t(b_t-n)^2$. Finally require $4V_p(n)>pn(p-n)$
for every mean $1\le n<p$ that occurs. The procedure treats $297$ primes,
the last being $2039$.

The independent check uses the value multiplicities
$N_{\varepsilon,\delta}(j)=1+\varepsilon\chi(j-\delta)$ of
\eqref{eq:quadratic-law} instead of iterating over $t$. Let $M_1,M_2$ be the
sums of the values and of their squares for a given $\delta$, starting
with $\delta=0$, and let $z$ be the multiplicity of the value $p-1$.
Passing from $\delta$ to $\delta+1$ increases every value by one, except
that the $z$ values equal to $p-1$ become $0$. Hence
\[
 M_1'=M_1+p-pz,\qquad
 M_2'=M_2+2M_1+p-p^2z,
\]
where the old value of $M_1$ is used in the second formula. The minima and
the strict margins obtained from this recurrence agree with those of the
direct loops.

\subsection{Orbit and stabilizer verification}
For each row of Table~\ref{tab:representatives}, take the set and its
complement. For every $(a,b,c,d)\in\F_p^4$ with $ad-bc\ne0$, apply the map
$(x,y)\mapsto(ax+by,cx+dy)$, record the image, and record whether it fixes
the representative. Require that the orbits of different representatives be
disjoint, and verify the orbit--stabilizer relation.

For coverage, determine the four special directions of each image by
counting points on lines, and retain the images whose special directions
include those defined by $x$, $y$ and $y-x$; the fourth is then defined by
$y-rx$ for some $r\ne0,1$. Every image has its quadratic center at the
origin, so the retained images are normalized sets. Require that they form
exactly the collection of full candidates found by the profile
enumeration.

Finally, count the full and empty lines of each representative $S$ and
compare $S$ with $-S$. For each representative of the second type in
Theorem~\ref{thm:reflection}, verify \eqref{eq:reflection-main}, the
uniqueness of its full and empty lines, and the central symmetry of
$S\setminus L$ and $S\cup L'$. These computations confirm the structural
theorem on the examples; they are not used in its proof.

\subsection{Use of AI tools in the research}\label{app:ai}
Generative AI tools were used in the research itself, under the author's
supervision. The large language models GPT-5.6 Sol (OpenAI),
Claude Fable 5 (Anthropic) and Grok 4.6 (xAI), coordinated through a
Codex-controlled adaptation of the Danus agent framework \cite{Danus},
were used for mathematical exploration, generation of conjectures and
proof strategies, intermediate derivations, and construction of examples.
The workflow decomposed the problem into subproblems, compared alternative
proof approaches, and used iterative self-critique. GPT-6 Astra and
GPT-6.1 Sol (OpenAI) were subsequently used for the structural extensions,
for checking derivations, and for the exact-computation workflow.
Taken together, this use included proposing polynomial and Fourier
arguments, constructing candidate profiles, and developing the
verification programs described in this appendix.

The mathematical claims are supported by the proofs and cited results.
The finite classification and prime-exclusion checks are reproduced by the
parameter domains, integer tests and orbit comparisons specified in this
appendix and implemented in the ancillary programs, not by regenerating
model responses.

%% file: bibliography.tex
\begingroup
\raggedright

\endgroup

%% file: main.bbl
\begin{thebibliography}{15}
\bibitem{ACW}
Sam Adriaensen, Bence Csajb\'ok, and Zsuzsa Weiner,
\emph{On equidistributed directions in finite affine planes},
arXiv:2609.40023v1 (2026),
\href{https://doi.org/10.48550/arXiv.2609.40023}{doi:10.48550/arXiv.2609.40023}.

\bibitem{ASW}
Sam Adriaensen, Tam\'as Sz\H onyi, and Zsuzsa Weiner,
\emph{Multisets with few special directions and small weight codewords
in Desarguesian planes},
Des. Codes Cryptogr. \textbf{94} (2026), no.~2, article~35,
\href{https://doi.org/10.1007/s10623-025-01777-8}{doi:10.1007/s10623-025-01777-8}.

\bibitem{Ball}
Simeon Ball,
\emph{The number of directions determined by a function over a finite
field},
J. Combin. Theory Ser. A \textbf{104} (2003), no.~2, 341--350,
\href{https://doi.org/10.1016/j.jcta.2003.09.006}{doi:10.1016/j.jcta.2003.09.006}.

\bibitem{BBBSS}
Aart Blokhuis, Simeon Ball, Andries E. Brouwer, Leo Storme, and
Tam\'as Sz\H onyi,
\emph{On the number of slopes of the graph of a function defined on a
finite field},
J. Combin. Theory Ser. A \textbf{86} (1999), no.~1, 187--196,
\href{https://doi.org/10.1006/jcta.1998.2915}{doi:10.1006/jcta.1998.2915}.

\bibitem{Chakraborty}
Sourav Chakraborty, Swarnalipa Datta, Pranjal Dutta, Arijit Ghosh,
and Swagato Sanyal,
\emph{Structure of sparse Boolean functions over Abelian groups,
and its application to testing},
arXiv:2406.18700v4 (2026),
\href{https://doi.org/10.48550/arXiv.2406.18700}{doi:10.48550/arXiv.2406.18700}.

\bibitem{Ghidelli}
Luca Ghidelli,
\emph{On rich and poor directions determined by a subset of a finite plane},
Discrete Math. \textbf{343} (2020), no.~5, article~111811,
\href{https://doi.org/10.1016/j.disc.2020.111811}{doi:10.1016/j.disc.2020.111811}.

\bibitem{GKS}
Luca Ghidelli, Gergely Kiss, and G\'abor Somlai,
\emph{Non-existence of sets with few special directions},
arXiv:2609.21779v1 (2026),
\href{https://doi.org/10.48550/arXiv.2609.21779}{doi:10.48550/arXiv.2609.21779}.

\bibitem{ILMP}
Alex Iosevich, Allen Liu, Azita Mayeli, and Jonathan Pakianathan,
\emph{Wavelet decomposition and bandwidth of functions defined on vector
spaces over finite fields},
Bull. Hellenic Math. Soc. \textbf{62} (2018), 80--93;
preprint arXiv:1601.03473v1,
\url{https://arxiv.org/abs/1601.03473v1}.

\bibitem{Cylinder}
Gergely Kiss, \'Ad\'am Mark\'o, Zolt\'an L\'or\'ant Nagy, and G\'abor Somlai,
\emph{Cylinder type and $p$-divisible sets in $\F_p^3$},
arXiv:2601.09910v1 (2026),
\href{https://doi.org/10.48550/arXiv.2601.09910}{doi:10.48550/arXiv.2601.09910}.

\bibitem{KissSomlai}
Gergely Kiss and G\'abor Somlai,
\emph{Special directions on the finite affine plane},
Des. Codes Cryptogr. \textbf{92} (2024), no.~9, 2587--2597,
\href{https://doi.org/10.1007/s10623-024-01404-y}{doi:10.1007/s10623-024-01404-y}.

\bibitem{Lev}
Vsevolod F. Lev,
\emph{Point distribution and perfect directions in $\F_p^2$},
Unif. Distrib. Theory \textbf{15} (2020), no.~2, 93--98,
\href{https://doi.org/10.2478/udt-2020-0012}{doi:10.2478/udt-2020-0012}.

\bibitem{Danus}
Jihao Liu, Guoxiong Gao, Zeming Sun, Bin Wu, Shurui Liu, Jiedong Jiang,
Haocheng Ju, Leheng Chen, Ronnie Cheng, Xiping Zhang, and Bin Dong,
\emph{Danus: Orchestrating mathematical reasoning agents with fact-graph memory},
arXiv:2607.06447v2 (2026),
\href{https://doi.org/10.48550/arXiv.2607.06447}{doi:10.48550/arXiv.2607.06447}.

\bibitem{MandeSanyal}
Nikhil S. Mande and Swagato Sanyal,
\emph{On parity decision trees for Fourier-sparse Boolean functions},
in FSTTCS 2020, Leibniz Int. Proc. Inform. \textbf{182} (2020), 29:1--29:16,
\href{https://doi.org/10.4230/LIPIcs.FSTTCS.2020.29}{doi:10.4230/LIPIcs.FSTTCS.2020.29}.

\bibitem{Redei}
L\'aszl\'o R\'edei,
\emph{L\"uckenhafte Polynome \"uber endlichen K\"orpern},
Birkh\"auser, Basel, 1970,
\href{https://doi.org/10.1007/978-3-0348-4006-4}{doi:10.1007/978-3-0348-4006-4}.

\bibitem{Szonyi}
Tam\'as Sz\H onyi,
\emph{On the number of directions determined by a set of points in an
affine Galois plane},
J. Combin. Theory Ser. A \textbf{74} (1996), no.~1, 141--146,
\href{https://doi.org/10.1006/jcta.1996.0042}{doi:10.1006/jcta.1996.0042}.
\end{thebibliography}
